\documentclass[12pt]{amsart}
\usepackage{amsfonts}
\usepackage{amsmath}
\usepackage{amssymb}
\usepackage{mathrsfs}
\usepackage{a4}
\usepackage{amscd}
\usepackage[T1]{fontenc}
\usepackage{lmodern}
\usepackage{tikz}
\usepackage[pdfstartview=FitH,colorlinks,linkcolor=black,bookmarks,bookmarksnumbered]{hyperref}

\usetikzlibrary{matrix}

\newcommand{\heute}{1 October 2014}

\numberwithin{equation}{section}

\theoremstyle{plain}
\newtheorem{theorem}[equation]{Theorem}

\newtheorem{lemma}[equation]{Lemma}

\newtheorem{proposition}[equation]{Proposition}

\theoremstyle{remark}

\newtheorem*{defn}{Definition}
\newtheorem*{rk}{Remark}

\newcommand{\enref}[1]{\textup{(\ref{enum:#1})}}
\newcommand{\dashTwo}[1]{\textup{(\ref{two}${}'$)}}

\newcommand{\ignore}[1]{}

\newcommand{\f}[1][p]{\mathbb{F}_{#1}}

\newcommand{\Hom}{\operatorname{Hom}}

\newcommand{\Id}{\operatorname{Id}}

\newcommand{\eps}{\varepsilon}

\DeclareMathOperator{\Bild}{Im}

\DeclareMathOperator{\Ext}{Ext}

\DeclareMathSymbol\normal{\mathrel}{AMSa}{"43}
\newenvironment{textmatrix}{\left(\begin{smallmatrix}}{\end{smallmatrix}\right)}

\newcommand{\Gro}[1]{Gr{\"o}b\-ner}

\tikzset{vertex/.style={circle,draw,thick,fill,inner sep=0pt, minimum size = 2mm}}
\tikzset{myarr/.style={->,thick,shorten >=5pt, shorten <=5pt}}
\tikzset{badge/.style={circle,draw,very thick,inner sep=0.75mm}}

\begin{document}

\title[Ext algebra of Mathieu group]{The Ext algebra of the principal 2-block of the Mathieu group $M_{11}$}
\author[D. J. Green]{David J. Green}
\email{david.green@uni-jena.de}
\author[S. A. King]{Simon A. King}
\email{simon.king@uni-jena.de}
\thanks{King was supported by DFG grant GR 1585/6-1, which also gave Green travel assistance.}
\address{Institute for Mathematics \\
University of Jena \\ D-07737 Jena \\
Germany}
\date{\heute}

\begin{abstract}
\noindent
We complete the calculation of the Ext algebra of the principal 2-block of the Mathieu group $M_{11}$, extending work of Benson--Carlson and of Pawloski -- and duplicating work of Generalov~\cite{Generalov:M11}\@.
\end{abstract}

\maketitle

\section{Introduction}
\noindent
One major problem in modular representation theory is to understand in how many different ways simple modules may be assembled to form an indecomposable module. This is essentially a cohomological question.

To be more precise, recall from \cite[\S2.6]{Benson:I} that if $M,N$ are left modules over a unital ring $\Lambda$, then $\Ext^1_{\Lambda}(N,M)$ classifies extensions of the form
\[
0 \rightarrow M \rightarrow X \rightarrow N \rightarrow 0 \, ,
\]
and more generally $\Ext^r_{\Lambda}(N,M)$ has an interpretation in terms of longer exact sequences. Moreover there is a Yoneda product 
\[
\Ext^r_{\Lambda}(N, M) \times \Ext^s_{\Lambda}(L,N) \rightarrow \Ext^{r+s}_{\Lambda}(L,M)
\]
which corresponds to splicing two such exact sequences together. If $\Lambda$ has only finitely many simple modules, then this Yoneda product ensures that
\[
\bigoplus_{\text{$S,T$ simple $\Lambda$-modules}} \Ext^*_{\Lambda} (S,T)
\]
is a unital (noncommutative) ring, which is called the Ext algebra of~$\Lambda$. If $\Lambda$ is the modular group algebra $kG$ of a finite group (or its principal block), then the Ext algebra contains the cohomology ring $H^*(G,k) = \Ext^*_{kG}(k,k)$ as a graded commutative subalgebra. But in comparison to the better known cohomology ring with its emphasis on the trivial module, the Ext algebra encodes information about all simple modules and treats them as equals.  

Of course, the trivial module is the only simple module for the modular group algebra of a $p$-group, and so for $p$-groups the cohomology ring and the Ext algebra coincide. But leaving this case aside,
to date there have been exceptionally few complete calculations of Ext algebras. Carlson, Green\footnote{E. L. Green, not the first author.} and Schneider~\cite{CaGrSchn:Ext} described a procedure to calculate the Ext algebra of a $p$-block of a finite group. For computational purposes it can be helpful to replace a block by its basic algebra, see~\cite[I.2]{Erdmann:habil} and \cite[p.~319]{CaGrSchn:Ext}\@. Lux devised methods to compute basic algebras~\cite{Lux:habil}\@. His student Hoffman computed several basic algebras~\cite{HoffmanT:thesis,HoffmanT:website}, and Pawloski -- another student of Lux -- used these results to compute several Ext algebras in low degrees~\cite{PawloskiR:thesis}\@. In several of Pawloski's computations it would appear that all generators and relations have been found, but no attempt is made to prove this.

In this paper we obtain in Theorem~\ref{thm:M11p2b1} the full structure of the Ext algebra of the principal $2$-block of the Mathieu group $M_{11}$. Note that Benson and Carlson obtained extensive partial results using their diagrammatic methods~\cite[Sect.~12]{BensonCarlson:Diagrammatic}\@. Specifically, there are three simple modules: the trivial module $K$, and modules $M,N$. Benson and Carlson compute the ring $\Ext^*(S,S)$ for $S=K,M,N$ and observe that $\Ext^*(M,M)$ is noncommutative.

\section{The basic algebra}
\noindent
The starting point for our computation is Hoffman's description of the basic
algebra~\cite{HoffmanT:website} of the principal block of $M_{11}$\@.  We
follow follow the Benson--Carlson naming convention for the simple modules:
the trivial module $K$, together with $M$~and $N$. In the $2$-block $M$ has
degree $44$ and $N$ has degree $10$, but of course $K,M,N$ are all
$1$-dimensional in the basic algebra.

We wrote a GAP~\cite{GAP4} script to extract a presentation of the basic algebra, as this is only implicit in Hoffman's data.
The basic algebra of the principal block of $\bar{\mathbb{F}}_2M_{11}$ is given by the quiver
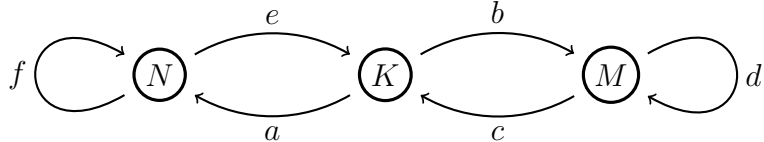
\[
\begin{tikzpicture}[every loop/.style={min distance=20mm}]
\node at (0,0) [badge] (A) {$K$};
\node at (-30mm,0) [badge] (B) {$N$};
\node at (30mm,0) [badge] (C) {$M$};
\draw[myarr] (B) to[out=30,in=150] node[auto] {$e$}  (A);
\draw[myarr] (A) to[out=30,in=150] node[auto] {$b$} (C);
\draw[myarr] (C) to[out=210,in=330] node[auto] {$c$} (A);
\draw[myarr] (A) to[out=210,in=330] node[auto] {$a$} (B);
\draw[myarr] (B) to[out=210,in=150,loop] node[auto] {$f$} ();
\draw[myarr] (C) to[out=30,in=330,loop] node[auto] {$d$} ();
\end{tikzpicture}
\]
and the reduced \Gro. basis of the relations ideal is
\begin{xalignat*}{4}
bd & = 0 & fe & = ebc & dc & = 0 & ea & = f^3 \\
af & = bca & cb & = 0 & d^2 & = caeb & aebc & = bcae \\
ebca & = f^4 & bcaeb & = 0 & f^5 & = 0
\end{xalignat*}
Here, shorter paths are more significant than longer ones, and paths of the same length are ordered lexicographically \emph{from the right}, with $a > b$. Note in particular that $aebc > bcae$. Finally, we may choose any ordering of the three vertices.
\medskip

\noindent
Since Morita equivalence leaves the Ext algebra unchanged we follow~\cite{CaGrSchn:Ext} and compute the Ext algebra of the principal block by computing the Ext algebra of the basic algebra.

\begin{theorem}
\label{thm:M11p2b1}
The Ext algebra of the principal block of $M_{11}$ at the prime $2$ has a
presentation with the following eight generators:
\begin{xalignat*}{4}
\alpha & \in \Ext^1(K,N)
&
\beta & \in \Ext^1(K,M)
&
\gamma & \in \Ext^1(M,K)
&
\delta & \in \Ext^1(M,M)
\\
\eps & \in \Ext^1(N,K)
&
\phi & \in \Ext^1(N,N)
&
\kappa & \in \Ext^4(K,K)
&
\nu & \in \Ext^4(N,N)
\end{xalignat*}
subject to the relations
\begin{xalignat*}{4}
\alpha \gamma & = 0
&
\gamma \beta & = 0
&
\eps \alpha & = 0
&
\beta \eps & = 0
\\
\phi^2 & = 0
&
\eps \phi \alpha & = 0
&
\phi \alpha \eps & = \alpha \eps \phi
&
\delta^2 \beta \gamma & = \beta \gamma \delta^2
\\
\gamma \delta^2 \beta & = 0
&
\kappa\gamma & = \gamma \delta^4
&
\delta^4 \beta & = \beta \kappa
&
\nu \alpha & = \alpha \kappa
\\
\kappa \eps & = \eps \nu
&
\nu \phi & = \phi \nu
\end{xalignat*}
These $14$ relations are minimal, and they form a \Gro. basis with
respect to the lexicographic (from the left) ordering $\alpha < \beta < \gamma
< \delta < \eps < \phi < \kappa < \nu$.
\end{theorem}

\begin{rk}
  By contrast, the cohomology ring $H^*(M_{11},\bar{\mathbb{F}}_2) \cong \Ext^*(K,K)$ has
  three generators and only one relation. But the generators are in degrees
  $3$, $4$ and $5$; and the relation is in degree ten.  The Ext Algebra has 8
  generators and 14 relations, but the last relation is in degree five.

The generators of $\Ext^*(K,K)$ are $\gamma \delta \beta$, $\kappa$ and $\gamma \delta^3 \beta$.
\end{rk}

\begin{rk}
As observed by Benson and Carlson, $\Ext^*(M,M)$ is noncommutative: for $\delta \cdot \beta \gamma \neq \beta \gamma \cdot \delta$.
\end{rk}

\begin{proof}
Proposition~\ref{prop:ExtM11-gens} constructs the generators and proves that they generate. Proposition~\ref{prop:ExtM11-rels} handles the relations.
\end{proof}

\subsection*{Projective indecomposable modules}
The projective indecomposable (right) module (PIM) for a vertex is the vector
space spanned by all paths starting at that vertex. The three PIMs and their
bases of standard monomials are:
\begin{equation}
\label{eqn:stdMons}
\begin{array}{l|l}
P_K & 1_K, a, b, ae, bc, aeb, bca, bcae \\
\hline
P_M & 1_M, c, d, ca, cae, caeb \\
\hline
P_N & 1_N, e, f, eb, f^2, ebc, f^3, f^4
\end{array}
\end{equation}
So in particular $P_K$ is spanned by the vertex path, together with all words
beginning with $a$ or $b$.

If $I_1, I_2,\ldots,I_s$ are vertices with projective indecomposables $P_{I_1}, P_{I_2}, \ldots, P_{I_s}$, then we abbreviate $P_{I_1}\oplus P_{I_2}\oplus \cdots \oplus P_{I_s}$ to
$P_{I_1 I_2\ldots I_s}$.  We write maps in matrix
notation. Hence, if $I_1, I_2,\ldots,I_m$ and $J_1, J_2,\ldots,J_n$ are vertices and $x_{i,j}$ are paths starting at $J_i$ and ending at
$I_j$, we define $P_{I_1I_2\ldots I_m}\xrightarrow{\begin{textmatrix}
    x_{i,j} \end{textmatrix}} P_{J_1J_2 \ldots J_n}$ by
\[
\begin{textmatrix} w_1\\\vdots\\ w_m\end{textmatrix} \mapsto
\begin{textmatrix} x_{1,1}w_1+\cdots + x_{1,m}w_m\\
  \vdots\\
  x_{n,1}w_1+\cdots + x_{n,m}w_m\end{textmatrix}
\]

\section{The minimal resolutions}
The minimal resolution of~$N$ is periodic of period four:
\[
\begin{CD}
\cdots @>>> P_{KN} @>{\begin{textmatrix}e&f\end{textmatrix}}>> P_N @>{f^4}>> P_N @>{\begin{textmatrix}a\\f\end{textmatrix}}>> P_{KN} @>{\begin{textmatrix}bc&a\\e&f^2\end{textmatrix}}>> P_{KN} @>{\begin{textmatrix}e&f\end{textmatrix}}>> P_N
\end{CD}
\]

\noindent
The minimal resolution of~$M$ is the total complex of the double complex which starts as follows and occupies the entire second quadrant:

\begin{small}
\begin{equation}
\label{eqn:minresM}
\begin{CD}
&&&&&&&&&&& \vdots \\
&&&&&&&&&&\cdots\quad& P_M \\
&&&&&&&&&\vdots&& @VV{b}V \\
&&&&&&&&\cdots\quad& P_M @>{aeb}>> P_K \\
&&&&&&&\vdots&& @VV{d}V @VV{c}V \\
&&&&&&\cdots\quad& P_K @>{cae}>> P_M @>{d}>> P_M \\
&&&&&\vdots&& @VV{c}V @VV{b}V @VV{d}V \\
&&&&\cdots\quad& P_M @>{d}>> P_M @>{aeb}>> P_K @>{cae}>> P_M \\
&&&\vdots&& @VV{aeb}V @VV{d}V @VV{c}V @VV{b}V \\
&&\cdots\quad& P_M @>{b}>> P_K @>{c}>> P_M @>{d}>> P_M @>{aeb}>> P_K \\
&\vdots&& @VV{d}V @VV{cae}V @VV{aeb}V @VV{d}V @VV{c}V \\
\cdots\quad & P_K @>{c}>> P_M @>{d}>> P_M @>{b}>> P_K @>{c}>> P_M @>{d}>> P_M
\end{CD}
\end{equation}
\end{small}

\noindent
Compare this with the resolution over the group algebra in~\cite[p.~106]{BensonCarlson:Diagrammatic}\@.
The minimal resolution of~$K$ starts as follows, compare \cite[p.~267]{BensonCarlson:multiple}:

\begin{small}
\begin{equation}
\label{eqn:minresK}
\begin{CD}
\vdots && \vdots && \vdots && \vdots && \vdots && \vdots \\
P_N @>{a}>> P_K @>{cae}>> P_M @>{d}>> P_M @>{aeb}>> P_K @>{cae}>> P_M
\\ &&&& @VV{eb}V @VV{d}V @VV{c}V @VV{b}V
\\ &&&& P_N @>{ca}>> P_M @>{d}>> P_M @>{aeb}>> P_K
\\ &&&& @VV{f}V @VV{b}V @VV{d}V @VV{c}V
\\ &&&& P_N @>{a}>> P_K @>{cae}>> P_M @>{d}>> P_M
\\ &&&&&&&& @VV{eb}V @VV{d}V
\\ &&&&&&&& P_N @>{ca}>> P_M
\\ &&&&&&&& @VV{f}V @VV{b}V
\\ &&&&&&&& P_N @>{a}>> P_K
\end{CD}
\end{equation}
\end{small}
\par
\noindent
In particular, for all $r \geq 0$ there is a copy of $P_N$ at position $(2r+1,2r)$, and nothing at position $(2r+2,2r+1)$.
\medskip

\noindent
Minimality in both cases is immediate, as all maps are zero modulo
radical.\footnote{The radical is spanned by paths of lengths at least one.} We
demonstrate that they are resolutions by constructing contracting homotopies.

\begin{defn}
  If $P \xrightarrow{x} Q$ is one of $P_M \xrightarrow{b} P_K$, $P_K
  \xrightarrow{c} P_M$, $P_M \xrightarrow{d} P_M$, $P_M \xrightarrow{eb} P_N$
  and $P_N \xrightarrow{f} P_N$, we define $Q \xrightarrow{h_x} P$ to be the
  vector space homomorphism given on the standard monomials $w$~in $Q$ by
\[
h_x(w) = \begin{cases} v & \text{$v$ a standard monomial and $w = xv$} \\ 0 & \text{otherwise} \end{cases} \, .
\]
Here $w = xv$ means equality in the basic algebra; as words they need not
coincide. For example, $h_d(caeb) = d$, since $caeb = d^2$.

Note that the map $h_x$ is well defined. This can be verified by inspecting
the Gr\"obner relations. These relations can be interpreted as replacement
rules, where the left hand side of each relation is replaced by the right hand
side. We see that there is no rule that would replace a word starting with
$b$, $c$, $d$ or $eb$ by a word that starts with a different sub-word.

The values of $h_x(w)$ can be read off of the following table.
\begin{equation}
\label{eqn:hx}
\begin{tabular}{c|l|l}
$x$ & $w = xv$ & $w$ with $h_x(w) = 0$ \\
\hline
$b$ & $b$, $bc$, $bca$, $bcae$ & $1_K$, $a$, $ae$, $aeb$ \\
$c$ & $c$, $ca$, $cae$, $caeb$ & $1_M$, $d$ \\
$d$ &  $d$, $caeb = d^2$ & $1_M$, $c$, $ca$, $cae$ \\
$eb$ & $eb$, $ebc$, $f^4 = ebca$ & $1_N$, $e$, $f$, $f^2$, $f^3$ \\
$f$ & $f$, $ebc = fe$, $f^2$, $f^3$, $f^4$ & $1_N$, $e$, $eb$ \\
\hline
\end{tabular}
\end{equation}
\end{defn}

\begin{lemma}
\label{lemma:bcd}
If $P \xrightarrow{x} Q \xrightarrow{y} R$ is one of $P_M \xrightarrow{b} P_K \xrightarrow{c} P_M$, $P_K \xrightarrow{c} P_M \xrightarrow{d} P_M$, $P_M \xrightarrow{d} P_M \xrightarrow{b} P_K$ and $P_M \xrightarrow{eb} P_N \xrightarrow{f} P_N$ then
\[
x \circ h_x + h_y \circ y = \Id_Q \, .
\]
Hence
$P_M \xrightarrow{eb} P_N \xrightarrow{f} P_N$ and the $3$-periodic sequence
\[
\begin{CD}
\cdots @>{c}>> P_M @>{d}>> P_M @>{b}>> P_K @>{c}>> P_M @>{d}>> P_M @>{b}>> \cdots
\end{CD}
\]
are exact.
\end{lemma}

\begin{proof}
By inspection in Table~\eqref{eqn:hx}\@. For example, the $v$ with $h_{eb}(v) = 0$ are precisely the $v$ such that $fv$ appears in the ``$w = xv$'' column for $x=f$. Exactness since in addition $feb = ebcb = 0$, and $cb = dc = bd = 0$.
\end{proof}

\noindent
Both double complexes are built from the following commutative squares and
their transposes. Note that these squares all correspond to relations in the
basic algebra.
\[
\begin{array}{c@{\qquad}c@{\qquad}c}
\begin{tikzpicture}[x=1cm,y=1cm]
  \matrix [matrix of math nodes,column sep=2mm,row sep=2mm]
  {
  |(A)| P_M & & |(B)| P_K \\ & |(X)| A & \\ |(C)| P_M & & |(D)| P_M \\
};
\begin{scope}[every node/.style={midway,auto,font=\scriptsize}]
\foreach \s/\t/\u in {A/B/aeb,B/D/c} \draw[->] (\s) -- node {$\u$} (\t); 
\foreach \s/\t/\u in {A/C/d,C/D/d} \draw[->] (\s) -- node[swap] {$\u$} (\t); 
\end{scope}
\end{tikzpicture}
&
\begin{tikzpicture}[x=1cm,y=1cm]
  \matrix [matrix of math nodes,column sep=2mm,row sep=2mm]
  {
  |(A)| P_K & & |(B)| P_M \\ & |(X)| B & \\ |(C)| P_M & & |(D)| P_K \\
};
\begin{scope}[every node/.style={midway,auto,font=\scriptsize}]
\foreach \s/\t/\u in {A/B/cae,B/D/b} \draw[->] (\s) -- node {$\u$} (\t); 
\foreach \s/\t/\u in {A/C/c,C/D/aeb} \draw[->] (\s) -- node[swap] {$\u$} (\t); 
\end{scope}
\end{tikzpicture}
&
\begin{tikzpicture}[x=1cm,y=1cm]
  \matrix [matrix of math nodes,column sep=2mm,row sep=2mm]
  {
  |(A)| P_M & & |(B)| P_M \\ & |(X)| C & \\ |(C)| P_K & & |(D)| P_M \\
};
\begin{scope}[every node/.style={midway,auto,font=\scriptsize}]
\foreach \s/\t/\u in {A/B/d,B/D/d} \draw[->] (\s) -- node {$\u$} (\t); 
\foreach \s/\t/\u in {A/C/b,C/D/cae} \draw[->] (\s) -- node[swap] {$\u$} (\t); 
\end{scope}
\end{tikzpicture}
\\
\begin{tikzpicture}[x=1cm,y=1cm]
  \matrix [matrix of math nodes,column sep=2mm,row sep=2mm]
  {
  |(A)| P_N & & |(B)| P_M \\ & |(X)| D & \\ |(C)| P_N & & |(D)| P_K \\
};
\begin{scope}[every node/.style={midway,auto,font=\scriptsize}]
\foreach \s/\t/\u in {A/B/ca,B/D/b} \draw[->] (\s) -- node {$\u$} (\t); 
\foreach \s/\t/\u in {A/C/f,C/D/a} \draw[->] (\s) -- node[swap] {$\u$} (\t); 
\end{scope}
\end{tikzpicture}
&
\begin{tikzpicture}[x=1cm,y=1cm]
  \matrix [matrix of math nodes,column sep=2mm,row sep=2mm]
  {
  |(A)| P_M & & |(B)| P_M \\ & |(X)| E & \\ |(C)| P_N & & |(D)| P_M \\
};
\begin{scope}[every node/.style={midway,auto,font=\scriptsize}]
\foreach \s/\t/\u in {A/B/d,B/D/d} \draw[->] (\s) -- node {$\u$} (\t); 
\foreach \s/\t/\u in {A/C/eb,C/D/ca} \draw[->] (\s) -- node[swap] {$\u$} (\t); 
\end{scope}
\end{tikzpicture}
\end{array}
\]

\begin{lemma}
\label{Y-OK}
$\begin{array}{c}
\begin{tikzpicture}[x=1cm,y=1cm]
  \matrix [matrix of math nodes,column sep=2mm,row sep=2mm]
  {
  |(A)| R & & |(B)| P \\ & |(X)| X & \\ |(C)| S & & |(D)| Q \\
};
\begin{scope}[every node/.style={midway,auto,font=\scriptsize}]
\draw[->] (A) -- node {$\alpha$} (B.west|-A); 
\draw[->] (A) -- node[swap] {$y$} (C.north-|A); 
\draw[->] (B) -- node {$x$} (D.north-|B); 
\draw[->] (C) -- node[swap] {$\beta$} (D.west|-C);
\end{scope}
\end{tikzpicture}\end{array}
$ satisfies $\alpha h_y = h_x \beta$ for $X = A, B, C, D, E$.
\end{lemma}

\begin{proof}
Both maps send $y$ to $\alpha$, and for $X=D$ they both map $fe = ebc$ to $cae$. By~\eqref{eqn:hx} they kill all remaining monomials.
\end{proof}

\begin{defn}
In $\begin{array}{c}
\begin{tikzpicture}[x=1cm,y=1cm]
  \matrix [matrix of math nodes,column sep=2mm,row sep=2mm]
  {
  |(A)| R & & |(B)| P \\ & |(X)| X & \\ |(C)| S & & |(D)| Q \\
};
\begin{scope}[every node/.style={midway,auto,font=\scriptsize}]
\draw[->] (A) -- node {$\alpha$} (B.west|-A); 
\draw[->] (A) -- node[swap] {$y$} (C.north-|A); 
\draw[->] (B) -- node {$x$} (D.north-|B); 
\draw[->] (C) -- node[swap] {$\beta$} (D.west|-C);
\end{scope}
\end{tikzpicture}\end{array}
$ for $X = A, C, D$ define $h'_X \colon Q \rightarrow S$ to be the linear map given on standard monomials by
\[
h'_X(w) = \begin{cases} v & \text{$w = \beta v$ and $w \notin \Bild(x)$} \\ 0 & \text{otherwise} \end{cases} \, .
\]
Again, $w = \beta v$ means equality in the basic algebra. The table lists all cases.
\begin{equation}
\label{eqn:hdash}
\begin{tabular}{c|c|c|c|c|l|l}
$X$ & $\beta$ & $Q$ & $S$ & $x$ & $w = \beta v \notin \Bild(x)$ & $w$ with $h'_X(w) = 0$ \\
\hline
$A$ & $d$ & $P_M$ & $P_M$ & $c$ & $d$ & $1_M$, $c$, $ca$, $cae$, $caeb$ \\
$C$ & $cae$ & $P_M$ & $P_K$ & $d$ & $cae$ & $1_M$, $c$, $d$, $ca$, $caeb = d^2$ \\
$D$ & $a$ & $P_K$ & $P_N$ & $b$ &  $a$, $ae$, $aeb$ & $1_K$, $b$, $bc$, $bca$, $bcae$ \\
\hline
\end{tabular}
\end{equation}
\end{defn}

\begin{lemma}
\label{lemma:hdash}
Let $\begin{array}{c}
\begin{tikzpicture}[x=1cm,y=1cm]
  \matrix [matrix of math nodes,column sep=2mm,row sep=2mm]
  {
  |(A)| R & & |(B)| P \\ & |(X)| X & \\ |(C)| S & & |(D)| Q \\
};
\begin{scope}[every node/.style={midway,auto,font=\scriptsize}]
\draw[->] (A) -- node {$\alpha$} (B.west|-A); 
\draw[->] (A) -- node[swap] {$y$} (C.north-|A); 
\draw[->] (B) -- node {$x$} (D.north-|B); 
\draw[->] (C) -- node[swap] {$\beta$} (D.west|-C);
\end{scope}
\end{tikzpicture}\end{array}
$be one of $A, C, D$. Then
\begin{enumerate}
\item
\label{enum:hdash-1}
$h'_X \circ x = 0$.
\item
\label{enum:hdash-2}
$(\beta h'_X + x h_x)(w) = \begin{cases} w & w \in \Bild(x) + \Bild(\beta) \\ 0 & \text{otherwise} \end{cases}$
for  $w \in Q$ standard monomial.
\end{enumerate}
\end{lemma}

\begin{proof}
Immediate from the construction.
\end{proof}

\begin{proposition}
\label{prop:minres}
\eqref{eqn:minresM} is a minimal resolution of~$M$ with contracting homotopy
\begin{small}
\begin{equation}
\label{eqn:contractM}
\begin{CD}
&&&&&&&&&&& \vdots \\
&&&&&&&&&&& P_M \\
&&&&&&&&&\vdots&& @A{h_b}AA \\
&&&&&&&&& P_M && P_K \\
&&&&&&&\vdots&& @A{h_d}AA @A{h_c}AA \\
&&&&&&& P_K && P_M && P_M \\
&&&&&\vdots&& @A{h_c}AA @A{h_b}AA @A{h_d}AA \\
&&&&\cdots\quad& P_M @<{h'_A}<< P_M && P_K && P_M \\
&&&&&&& @A{h'_A}AA @A{h_c}AA @A{h_b}AA \\
&&\cdots\quad& P_M @<{h_b}<< P_K @<{h_c}<< P_M @<{h'_A}<< P_M && P_K \\
&&& && && && @A{h'_A}AA @A{h_c}AA \\
\cdots\quad & P_K @<{h_c}<< P_M @<{h_d}<< P_M @<{h_b}<< P_K @<{h_c}<< P_M @<{h'_A}<< P_M
\end{CD}
\end{equation}
\end{small}
\par \noindent
In particular, the zig-zag pattern at positions $(r,r)$ and $(r+1,r)$ continues indefinitely. Also, \eqref{eqn:minresK} is a minimal resolution of~$K$ with contracting homotopy
\begin{small}
\begin{equation}
\label{eqn:contractK}
\begin{CD}
\vdots && \vdots && \vdots && \vdots && \vdots && \vdots \\
P_N @<{h'_D}<< P_K @<{h'_C}<< P_M && P_M && P_K && P_M
\\ &&&& @A{h_{eb}}AA @A{h_d}AA @A{h_c}AA @A{h_b}AA
\\ &&&& P_N && P_M && P_M && P_K
\\ &&&& @A{h_f}AA @A{h_b}AA @A{h_d}AA @A{h_c}AA
\\ &&&& P_N @<{h'_D}<< P_K @<{h'_C}<< P_M && P_M
\\ &&&&&&&& @A{h_{eb}}AA @A{h_d}AA
\\ &&&&&&&& P_N && P_M
\\ &&&&&&&& @A{h_f}AA @A{h_b}AA
\\ &&&&&&&& P_N @<{h'_D}<< P_K
\end{CD}
\end{equation}
\end{small}
\par\noindent
As in \eqref{eqn:minresK}, there is a $P_N$ at each $(2r+1,2r)$ and nothing at each $(2r+2,2r+1)$.
\end{proposition}

\begin{proof}
The double complexes have the structure
\begin{xalignat*}{2}
M & \colon \begin{array}{llllll}
\ddots & \ddots & \ddots & \ddots & \ddots & \vdots \\
\ddots & B   & C   & A   & B   & C \\
\ddots & A'   & B   & C   & A   & B \\
\ddots & (A')^T & A'   & B   & C   & A \\
\ddots & B^T & (A')^T & A'   & B   & C \\
\ddots & C^T & B^T & (A')^T & A'   & B \\
\cdots\vphantom{\ddots}  & A^T & C^T & B^T & (A')^T & A'
\end{array}
&
K & \colon
\begin{array}{llllll}
\ddots & \ddots & \ddots & \ddots & \ddots & \vdots \\
& E & A & B & C & A \\
& D' & C' & A & B & C \\
& & & E & A & B \\
& & & D' & C' & A \\
& & & & & E \\
& & & & & D'
\end{array}
\end{xalignat*}
where the dashes only pertain to the construction of the homotopy.

We now construct the homotopy~$h$. Every term in the double complexes occurs as~$Q$ in a square $\begin{array}{c}
\begin{tikzpicture}[x=1cm,y=1cm]
  \matrix [matrix of math nodes,column sep=2mm,row sep=2mm]
  {
  |(A)| R & & |(B)| P \\ & |(X)| X & \\ |(C)| S & & |(D)| Q \\
};
\begin{scope}[every node/.style={midway,auto,font=\scriptsize}]
\draw[->] (A) -- node {$\alpha$} (B.west|-A); 
\draw[->] (A) -- node[swap] {$y$} (C.north-|A); 
\draw[->] (B) -- node {$x$} (D.north-|B); 
\draw[->] (C) -- node[swap] {$\beta$} (D.west|-C);
\end{scope}
\end{tikzpicture}\end{array}
$, except for~$S$ in $X=D,E$. So it suffices to define $h \vert_Q$ for each~$X$, together with $h\vert_S$ for $X = D,E$. For $X = A,B,C,D,E$ we define $h\vert_Q$ by
\[
h\vert_Q = \begin{cases} h_x & \text{in $X$} \\ h_x + h'_X & \text{in $X'$} \end{cases} \, .
\]
For $X^T$ we construct $h\vert_Q$ and then transpose. For $X = D,E$ we set $h \vert_S = h_y$.

To verify the homotopy property at $Q$ we need to consider the surrounding squares $\begin{array}{|c|c|}\hline\strut X & Y \\ \hline Z & \multicolumn{1}{c}{} \\ \cline{1-1} \end{array}$. Reserving the case $X=D'$ for later, it follows that $Y$ is not dashed, so Lemma~\ref{Y-OK} allows us to ignore it.

For $(X,Z) = (B,A)$, $(C,B)$, $(A,C)$, $(A,E)$ and $(E,D)$ the homotopy property is satisfied by Lemma~\ref{lemma:bcd} if $X$ is not dashed: Lemma~\ref{lemma:hdash}\,\enref{hdash-1} deals with the case where $Z$ is dashed. For $X=A'$ we have $(ch_c + dh'_A)(w) = \begin{cases} 0 & w = 1_M \\ w & \text{otherwise} \end{cases}$ by Lemma~\ref{lemma:hdash}\,\enref{hdash-2}\@. That covers the bottom right $A'$ in the resolution of~$M$, and for $(X,Z) = (A',(A')^T)$ we just need to add $h'_A d + (h_c d + aeb h'_A)$: but $h_c d + aeb h'_A = 0$ by inspection, whereas $h'_A d$ maps $1_M$ to itself and kills all other standard monomials.

For $X = C'$ we need $d h_d + cae h'_C + h_{eb}eb = \Id$. Now, $d h_d$ maps $d, caeb=d^2$ to themselves and kills everything else; $h_{eb} eb$ is the identity on $1_M, c, ca$ and zero elsewhere; and $cae h'_C$ preserves $cae$ and annihilates everything else. So the sum is the identity.

For $X = D'$ we have $(a h'_A + b h_B)(w) = \begin{cases} 0 & w = 1_K \\ w & \text{otherwise} \end{cases}$ by Lemma~\ref{lemma:hdash}\,\enref{hdash-2}\@. So the resolution of~$K$ is exact in degree zero. In higher degrees we can ignore the maps to the top right hand corner of $Y = C'$ by Lemma~\ref{Y-OK}, but we do need to add $h'_C cae$. This preserves $1_K$ and kills everything else, so $a h'_A + b h_B + h'_C cae = \Id$.

For $S = P_N$ in $D'$ we have $f h_f + h'_D a = \Id$ by inspection, similarly for $S = P_N$ in $E$ we have $eb h_{eb} + h_f f = \Id$. In both cases the maps to the top right hand corner of the square cancel each other out, by Lemma~\ref{Y-OK}\@.
\end{proof}

\section{The Ext algebra}
We recall how Ext groups may be read off from the minimal resolution. Let $A$ be a finite dimensional quotient of a algebra over the field $k$, and let $S,T$ be two vertex simples. If $Q_* \rightarrow S$ a minimal projective resolution, then minimality implies that every differential $\Hom_A(Q_n, T) \rightarrow \Hom_A(Q_{n+1},T)$ vanishes, and hence $\Ext^n(S,T) \cong \Hom_A(Q_n,T)$. Moreover, $Q_n$ is a direct sum of projective indecomposable modules, each of which corresponds to a vertex simple; and $\Hom_A(Q_n,T)$ is $k^r$, where  $r$ is the number of copies of $P_T$ in $Q_n$.

For example, we see from \eqref{eqn:minresM} that the minimal resolution $Q_*$ of $M$ has $Q_3 \cong P_M^3 \oplus P_K$. Hence $\Ext^3(M,M)$ is three-dimensional; $\Ext^3(M,K)$ is one-dimensional; and $\Ext^3(M,M) = 0$. Each of the three copies of $P_M$ in $Q_3$ corresponds to one basis vector of $\Ext^3(M,M)$.
\medskip

\noindent
For $L = K,M,N$ write $D(L)_*$ for the minimal resolution of~$L$\@. For $L=M,K$ we have $D(L)_n = \bigoplus_{r+s=n} D(L)_{rs}$, where $r$ is the horizontal degree and $s$ the vertical degree.

\begin{proposition}
\label{prop:ExtM11-gens}
The Ext algebra is generated by the following eight classes:
\begin{xalignat*}{2}
\alpha & \colon D(K)_{10} = P_N \twoheadrightarrow N \in \Ext^1(K,N)
&
\beta & \colon D(K)_{01} = P_M \twoheadrightarrow M \in \Ext^1(K,M)
\\
\gamma & \colon D(M)_{01} = P_K \twoheadrightarrow K \in \Ext^1(M,K)
&
\delta & \colon D(M)_{10} = P_M \twoheadrightarrow M \in \Ext^1(M,M)
\\
\eps & \colon D(N)_1 = P_{KN} \stackrel{\pi_1}{\twoheadrightarrow} K \in \Ext^1(N,K)
&
\phi & \colon D(N)_1 = P_{KN} \stackrel{\pi_2}{\twoheadrightarrow} N \in \Ext^1(N,N)
\\
\kappa & \colon D(K)_{22} = P_K \twoheadrightarrow K \in \Ext^4(K,K)
&
\nu & \colon D(N)_4 = P_N \twoheadrightarrow N \in \Ext^4(N,N)
\end{xalignat*}
\end{proposition}

\noindent
We of course view the idempotents in $\Ext^0(K,K)$, $\Ext^0(M,M)$ and $\Ext^0(N,N)$ as part of the structure.

\begin{proof}
Each projective indecomposable module in each of the three minimal resolutions corresponds to a basis vector of an Ext group. We need to show that every basis vector (except in degree zero) factors through one of the above eight classes. As we are dealing with the Yoneda product, we lift each of the eight classes to a chain map.
\begin{itemize}
\item
$D(M)_{{*}+1} \xrightarrow{\delta_*} D(M)_*$ shifts one column to the right (killing the right hand column in the process) and then transposes. So it is the identity map $D(M)_{r+1,s} \rightarrow D(M)_{s,r}$. Hence the class corresponding to each projective indecomposable in $D(M)_{{}>0,{*}}$ lies in the left ideal generated by~$\delta$\@. Observe that $\delta^2_*$ is the identity map $D(M)_{r+1,s+1} \rightarrow D(M)_{r,s}$.
\item
$D(M)_{{*}+1} \xrightarrow{\gamma_*} D(K)_*$ kills $D(M)_{r,s+1}$ for $s < r-1$ and shifts one row down.  For $s \geq r+1$ it is the identity map $D(M)_{r,s+1} \rightarrow D(K)_{r,s}$. The remaining cases:
\begin{align*}
D(M)_{2s,2s+1} = P_K & \xrightarrow{\Id} P_K = D(K)_{2s,2s}
\\
D(M)_{2s+1,2s+1} = P_M & \xrightarrow{eb} P_N = D(K)_{2s+1,2s}
\\
D(M)_{2s,2s} = P_M & \xrightarrow{0} 0 = P(K)_{2s,2s-1}
\\
D(M)_{2s-1,2s} = P_K & \xrightarrow{e} P_N = D(K)_{2s-1,2s-1}
\end{align*}

\noindent
So the left ideal generated by~$\gamma$ contains the classes corresponding to each projective indecomposable in $D(M)_{r,s}$ for $s \geq r+2$, and in $D(M)_{2s,2s+1}$. So the left ideal generated by $\gamma$~and $\delta$ is the whole of $\Ext^{{}>0}(M,\_)$.
\item
$D(K)_{{*}+1} \xrightarrow{\alpha_*} D(N)_*$ kills the right-hand edge of $D(K)_{**}$ and is zero except on the left-hand edge

\begin{small}
\[
\begin{CD}
\vdots \\ @VV{f}V 
\\ P_N @>{a}>> P_K @>{cae}>> P_M
\\ &&&& @VV{eb}V
\\ &&&& P_N
\\ &&&& @VV{f}V
\\ &&&& P_N
\end{CD}
\]
\end{small}
On this quotient it is given by
\[
\begin{CD}
\cdots @>{f}>> P_N @>{a}>> P_K @>{cae}>> P_M @>{eb}>> P_N @>{f}>> P_N
\\ && @V{\Id}VV @V{ebc}VV @V{\begin{textmatrix}b\\0\end{textmatrix}}VV @V{\begin{textmatrix}0\\\Id\end{textmatrix}}VV @V{\Id}VV
\\ \cdots @>{\begin{textmatrix}e&f\end{textmatrix}}>> P_N @>{f^4}>> P_N @>{\begin{textmatrix}a\\f\end{textmatrix}}>> P_{KN} @>{\begin{textmatrix}bc&a\\e&f^2\end{textmatrix}}>> P_{KN} @>{\begin{textmatrix}e&f\end{textmatrix}}>> P_N 
\end{CD}
\]
So all classes coming from $D(K)_{1+2n,2n}$ and $D(K)_{1+2n,1+2n}$ lie in the left ideal generated by~$\alpha$\@.
\item
$D(K)_{{*}+1} \xrightarrow{\beta_*} D(M)_*$ shifts one row down, transposes, and then maps to $D(M)_*$. For $s \geq r$ it is the identity map $D(K)_{r,s+1} \rightarrow D(M)_{s,r}$. The remaining cases:
\begin{align*}
D(K)_{2s+1,2s+1} = P_N & \xrightarrow{a} P_K = D(M)_{2s,2s+1}
\\
D(K)_{2s,2s} = P_K & \xrightarrow{ae} P_K = D(M)_{2s-1,2s}
\\
D(K)_{2s+1,2s} = P_N & \xrightarrow{0} P_M = D(M)_{2s-1,2s+1} \, .
\end{align*}
All classes coming from the region $s > r$ of $D(K)_{r,s}$ lie in the left ideal generated by~$\beta$\@.
\item
$\kappa_*$ is $D(K)_{r+2,s+2} \xrightarrow{\Id} D(K)_{r,s}$ for all~$r,s$\@. The left ideal generated by~$\kappa$ contains all classes from $D(K)_{**}$, except those coming from the first two columns.
So $\alpha$, $\beta$~and $\kappa$ generate $\Ext^{{}>0}(K,\_)$ as a left ideal.
\item
For $D(N)_{{*}+1} \xrightarrow{\eps_*} D(K)_*$ we define $E_*$ to be the (slightly thickened) left hand edge of $D(K)_*$, the periodic complex assembled from copies of
\[
\begin{CD}
\\ P_K @>{cae}>> P_M
\\ && @VV{eb}V
\\ && P_N @>{ca}>> P_M
\\ && @VV{f}V @VV{b}V
\\ && P_N @>{a}>> P_K
\end{CD}
\]
with the bottom right hand $P_K$ in bidegree $(2s,2s)$ ($s \geq 0$)\@. Then $E_*$ is a quotient of the complex $D(K)_*$, and there is a chain map $D(N)_{{*}+1} \rightarrow E_*$ as follows:
\[
\begin{CD}
\cdots @>{\begin{textmatrix}bc&a\\e&f^2\end{textmatrix}}>> P_{KN} @>{\begin{textmatrix}e&f\end{textmatrix}}>> P_N @>{f^4}>> P_N @>{\begin{textmatrix}a\\f\end{textmatrix}}>> P_{KN} @>{\begin{textmatrix}bc&a\\e&f^2\end{textmatrix}}>> P_{KN}
\\ && @V{\begin{textmatrix}\Id&0\end{textmatrix}}VV @V{ca}VV @V{\Id}VV @V{\begin{textmatrix}c&0\\0&\Id\end{textmatrix}}VV @V{\begin{textmatrix}\Id&0\end{textmatrix}}VV
\\ \cdots @>{\begin{textmatrix}b&a\end{textmatrix}}>> P_K @>{cae}>> P_M @>{eb}>> P_N @>{\begin{textmatrix}ca\\f\end{textmatrix}}>> P_{MN} @>{\begin{textmatrix}b&a\end{textmatrix}}>> P_K 
\end{CD}
\]
Now, the only maps out of $E_*$ into the rest of $D(K)_*$ is a $P_M \xrightarrow{d} P_M$ emanating horizontally from each~$P_M$, so since $dc=0$ the chain map lifts unaltered to $D(N)_{{*}+1} \xrightarrow{\eps_*} D(K)_*$.
\item
$D(N)_{{*}+1} \xrightarrow{\phi_*} D(N)_*$ is
\[
\begin{CD}
\cdots @>{\begin{textmatrix}bc&a\\e&f^2\end{textmatrix}}>> P_{KN} @>{\begin{textmatrix}e&f\end{textmatrix}}>> P_N @>{f^4}>> P_N @>{\begin{textmatrix}a\\f\end{textmatrix}}>> P_{KN} @>{\begin{textmatrix}bc&a\\e&f^2\end{textmatrix}}>> P_{KN}
\\ && @V{\begin{textmatrix}0&\Id\end{textmatrix}}VV @V{f^3}VV @V{\begin{textmatrix}0\\\Id\end{textmatrix}}VV @V{\begin{textmatrix}\Id&0\\0&f\end{textmatrix}}VV @V{\begin{textmatrix}0&\Id\end{textmatrix}}VV
\\ \cdots @>{\begin{textmatrix}e&f\end{textmatrix}}>> P_N @>{f^4}>> P_N @>{\begin{textmatrix}a\\f\end{textmatrix}}>> P_{KN} @>{\begin{textmatrix}bc&a\\e&f^2\end{textmatrix}}>> P_{KN} @>{\begin{textmatrix}e&f\end{textmatrix}}>> P_N
\end{CD}
\]
\item
$\nu_*$ is $D(N)_{r+4} \xrightarrow{\Id} D(N)_r$ for all $r$\@.
\end{itemize}
So the left ideal generated by $\eps$, $\phi$~and $\nu$ is the whole of $\Ext^{{}>0}(N,\_)$.
\end{proof}

\begin{proposition}
\label{prop:ExtM11-rels}
The relations are as described in Theorem~\ref{thm:M11p2b1}\@.
\end{proposition}

\begin{proof}
As the groups $\Ext^2(M,N)$, $\Ext^2(K,K)$ and $\Ext^2(N,M)$ are all zero, the products $\alpha \gamma$, $\gamma \beta$, $\eps \alpha$ and $\beta \eps$ vanish. Since $\phi_0 \circ \phi_1 \colon P_{KN} \rightarrow P_N$ is $\begin{textmatrix}0 & \Id\end{textmatrix} \begin{textmatrix}
\Id&0\\0&f\end{textmatrix} = \begin{textmatrix}
0 & f
\end{textmatrix}$ with image in the radical, $\phi^2 =0$. Similarly, $\phi_1 \alpha_2 = \begin{textmatrix}\Id&0\\0&f\end{textmatrix} \begin{textmatrix}b\\0\end{textmatrix}$ takes values in the radical, so $\eps \phi \alpha = 0$.

We have $\phi_0 \alpha_1\eps_2 = \begin{textmatrix}0 & \Id\end{textmatrix} \begin{textmatrix}0 \\ \Id\end{textmatrix} \Id = \Id_{P_N}$ and $\alpha_0 \eps_1 \phi_2 = \begin{textmatrix}0&\Id\end{textmatrix} \begin{textmatrix}c & 0 \\ 0 & \Id \end{textmatrix} \begin{textmatrix} 0 \\ \Id \end{textmatrix} = \Id$, so $\alpha \eps \phi = \phi \alpha \eps$.

$\beta_0 \gamma_1$ is $\begin{CD} P_{MMK} @>{\begin{textmatrix}\Id&0&0\\0&eb&0\end{textmatrix}}>> P_{MN} @>{\begin{textmatrix}\Id&0\end{textmatrix}}>> P_M\end{CD}$, that is $\begin{CD} P_{MMK} @>{\begin{textmatrix}\Id&0&0\end{textmatrix}}>> P_M\end{CD}$; and
$\beta_2 \gamma_3$ is $\begin{CD} P_{KMMKM} @>{\begin{textmatrix}\Id&0&0&0&0\\0&\Id&0&0&0\end{textmatrix}}>> P_{KM} @>{\begin{textmatrix}0&0\\0&\Id\\\Id&0\end{textmatrix}}>> P_{MMK} \end{CD}$. So since each lift of $\delta^2$ acts by discarding the first and last entries, we have $\delta^2 \beta \gamma = \beta \gamma \delta^2$.

$\beta_3$ is $\begin{CD} P_{MMK} @>{\begin{textmatrix}0&0&0\\0&0&ae\\0&\Id&0\\\Id&0&0\end{textmatrix}}>> P_{MKMM} \end{CD}$, so $(\delta^2 \beta)_1$ is  $\begin{CD} P_{MMK} @>{\begin{textmatrix}0&0&ae\\0&\Id&0\end{textmatrix}}>> P_{KM} \end{CD}$. So since $\gamma_0$ is projection onto the first summand $P_K$, the image lies in the radical and $\gamma \delta^2 \beta = 0$.

For each module $L = M,N,K$ there is a class in $\Ext^4(L,L)$ which describes the self-similarity of the minimal resolution: the periodicity class $\nu \in \Ext^4(N,N)$ as well as the classes $\delta^4 \in \Ext^4(M,M)$ and $\kappa \in \Ext^4(K,K)$. The remaining five relations describe how the six degree one generators commute with these degree four classes: the sixth relation is the fact that $\delta$ commutes with $\delta^4$, but that is trivial. All five relations are demonstrated in the same way, so we just consider $\kappa \gamma = \gamma \delta^4$. The map $\gamma_4$ is
$\begin{CD} P_{MMKMMK} @>{\begin{textmatrix}\Id&0&0&0&0&0\\0&\Id&0&0&0&0\\0&0&\Id&0&0&0\end{textmatrix}}>> P_{MMK} \end{CD}$, so $\kappa \gamma$ picks out the third summand of $P_{MMKMMK}$. And $\gamma \delta^4$ first discards the first, second, fifth and sixth summands before selecting the first summand of $P_{KM}$. So $\kappa \gamma = \gamma \delta^4$.

That establishes the relations. Now to show that they suffice. Observe that the leading term of each relation agrees with the lexicographic ordering $\alpha < \beta < \gamma < \delta < \eps < \phi < \kappa < \nu$. So as we have already established that the eight classes generate the Ext algebra, it only remains to check that the standard monomials with respect to these relations yield the same Hilbert-Poincar{\'e} series as the three minimal resolutions. From the resolutions we see that $\Ext^*(N,\_)$ has Hilbert-Poincar{\'e} series $\frac{1+2t+2t^2+t^3}{1-t^4}$; $\Ext^*(M,\_)$ has $\frac{1}{(1-t)^2}$; and $\Ext^*(K,\_)$ has $\frac{1+t}{(1-t)(1-t^4)}$. The standard monomials for $\Ext^*(N,\_)$ are $\nu^r$, $\eps \nu^r$, $\phi \nu^r$, $\alpha \eps \nu^r$, $\eps \phi \nu^r$, $\alpha \eps \phi \nu^r$ for $r \geq 0$, with Hilbert-Poincar{\'e} series $\frac{1+2t+2t^2+t^3}{1-t^4}$. The standard monomials for $\Ext^*(M,\_)$ are $w \delta^r$ for $r \geq 0$ and $w$ a right divisor of some $(\delta \beta \gamma)^n$: so the Hilbert-Poincar{\'e} series is $\frac{1}{(1-t)(1-t^4)}$. Finally the standard monomials for $\Ext^*(K,\_)$ are $w' \kappa^r$, $\alpha \kappa^r$, $\phi \alpha \kappa^r$, $w'' \delta^2 \beta \kappa^r$ for $w'$ a right divisor of some $(\gamma \delta \beta)^n$ and $w''$ a right divisor of some $(\beta \gamma \delta)^n$: so the Hilbert-Poincar{\'e} series is
\[
\frac{1+(t+t^2)(1-t)+t^3}{(1-t)(1-t^4)} = \frac{1+t}{(1-t)(1-t^4)} \, .
\]
This also proves that the relations of Proposition~\ref{prop:ExtM11-rels} are a \Gro. basis with respect to the stated lexicographic ordering.

If the relations were not minimal, then we could delete one and subsequently recover it using the 
Buchberger algorithm. But when one does this, most $S$-polynomials are identically zero: the exceptions are $\alpha \eps \phi \alpha$, $\phi \alpha \eps \phi$, $\alpha \eps \phi^2 \alpha$, $\beta \gamma \delta^2 \beta$, $\gamma \beta \gamma \delta^2$ and further relations in $\Ext^{{}\geq 6}$. The first two cannot help recovering a lost relation, as they have degree four and reduce to zero over the lower degree relations $\eps \phi \alpha = 0$, $\phi \alpha \eps = \alpha \eps \phi$ and $\phi^2 = 0$. Similarily, the next three have degree five but reduce to zero over $\gamma \beta = 0$, $\phi^2 = 0$ and $\gamma \delta^2 \beta = 0$. Finally, there is no relation in degree${} \geq 6$ that we could have deleted.
\end{proof}

\section*{Acknowledgements}
\noindent
This work was supported by the German Science
Foundation (DFG), project number GR 1585/6-1\@. We thank Klaus Lux for drawing our attention to Generalov's work.


\end{document}